\documentclass[11pt]{amsart}
\usepackage{amssymb,latexsym,epsfig,color}
\usepackage{graphicx}

\newcommand{\N}{{\mathbb N}}
\newcommand{\R}{{\mathbb R}}

\newcommand{\Set}{{\mathcal S}}

\newcommand{\NN}{{\mathcal N}}
\newcommand{\KK}{{\mathcal K}}
\newcommand{\PP}{{\mathcal P}}
\newcommand{\TT}{{\mathcal T}}

\newcommand{\red}{\textcolor{red}}

\newtheorem{theorem}{Theorem}[section]
\newtheorem{proposition}[theorem]{Proposition}
\newtheorem{corollary}[theorem]{Corollary}
\newtheorem{lemma}[theorem]{Lemma}

\theoremstyle{definition}
\newtheorem{definition}[theorem]{Definition}
\newtheorem{remark}[theorem]{Remark}

\title[relaxed complete partitions]
{relaxed complete partitions: \\ an error-correcting bachet's problem}
\author{Jorge Bruno \& Edwin O'Shea}
\thanks{JB is supported by an EMBARK fellowship of the Irish Research Council for Science, Engineering \& Technology. }
\address{School of Mathematics, NUI, Galway, 1 University Road, Galway, Ireland. }   
\address{Department of Mathematics and Statistics, James Madison University, Harrisonburg, VA, 22807.}
\email{brujo.email@gmail.com, osheaem@jmu.edu}
\date{\today}

\begin{document}

\begin{abstract} 
Motivated by an error-correcting generalization of Bachet's weights problem, 
we define and classify {\em relaxed complete partitions}. We show that these 
partitions enjoy a succinct description in terms of lattice points in polyhedra, with adjustments 
in the error being commensurate with translations in the defining hyperplanes. 
Our main result is that the enumeration of the minimal such partitions (those with fewest possible parts) 
is achieved via Brion's formula. This generalizes work of Park on classifying complete partitions and 
that of R{\o}dseth on enumerating minimal complete partitions. 
\end{abstract}

\maketitle

\vspace{-.1in}

\section{Introduction}

Recorded by Fibonacci \cite[On IIII Weights Weighing Forty Pounds]{Fib} in 1202, Bachet's 
problem ought to be regarded as one of the earliest problems in integer partitions. 
It asks: {\em what is the least number of pound weights that can be used on a scale pan to 
weigh any integral number of pounds from 1 to 40 inclusive, if the weights 
can be placed in either of the scale pans~?} Its solution consists of four parts and can be written 
as $40 = 1+3+9+27$ and is unique. Replacing 40 with any integer $m$, this problem has been generalized in a number of ways: 
by MacMahon \cite{MacQuar} in 1886; by Brown \cite{Bro} in 1961 and by Park \cite{Par} in 1998. 
The latter was the first to describe all possible solutions to Bachet's problem as originally stated, when {\em 40} is 
replaced with any integer $m$. A lively expository account of these various generalizations of Bachet's problem will appear 
in \cite{O'Sh}. 

We generalize Bachet's problem in a {\em relaxed} or {\em error-correcting} manner by considering 
the following variant of it (first asked by Tanton \cite{Tanton}) : 
{\em given a fixed but unknown integer weight $l$, weighing 
no more than 80 pounds, what is the least number of integer weights 
that can be used on a scale pan to discern $l$'s value, 
if the weights can be placed in either of the scale pans~?} 
Here, we still only need four parts and the partition $80 = 2+6+18+54$ will suffice: 
if the weight $l$ 
is even we can achieve its exact value with the parts of $2+6+18+54$ and a balanced scale and 
if the weight $l$ is odd then both $l-1$ and $l+1$ can be achieved by the parts of $2+6+18+54$ 
and the unknown weight can be seen to be heavier than the former and lighter than the latter. 
This is equivalent to saying the parts of $80 = 2+6+18+54$ can be used to weigh 
every integer between $1$ and $80$ on a two-scale pan, within an error of one. 
This leads us to the following definition. 
\begin{definition}
A partition $m = \lambda_0 + \lambda_1 + \cdots + \lambda_n$ with the parts in increasing order 
is an {\em $e$-relaxed $r$-complete partition} ({\em $(e,r)$-partition} for short) if 
no $e+1$ consecutive integers between $0$ and $rm$ are absent from the set 
$\{ \sum_{i=0}^{n} \alpha_i \lambda_i : \alpha_i \in \{ 0,1,\ldots,r \} \}$. We call the 
partition {\em minimal} if $n$ is as small as possible with this property.
\end{definition}

Park \cite{Par}, motivated by MacMahon's {\em perfect partitions} \cite{MacQuar}, 
called the $(0,r)$-partitions simply {\em $r$-complete partitions} and as a result of Park's work 
it can be shown that $40 = 1+3+9+27$ is the only minimal $2$-complete partition 
of $40$. This was also known to Hardy \& Wright \cite[\S 9.7]{HarWri}. 
To see the link between minimal $2$-complete partitions and Bachet's problem we only need observe 
that for any $2$-complete partition $m = \lambda_0 + \lambda_1 + \cdots + \lambda_n$, shifting the set 
$\{ \sum_{i=0}^{n} \alpha_i \lambda_i : \alpha_i \in \{ 0,1,2 \} \}$ by $-m$ we get the set 
$\{ \sum_{i=0}^{n} \beta_i \lambda_i : \beta_i \in \{ -1,0,1 \} \}$ which is exactly the set of weights 
achievable by the parts of $m = \lambda_0 + \lambda_1 + \cdots + \lambda_n$ by using both sides 
of the two-scale pan. Notice too that the minimal $1$-complete partitions are the variant of Bachet's 
problem where one can place weights on only one of the two scale pans.
Other simple observations can be made at this juncture. For example, if 
$\lambda_0 + \lambda_1 + \cdots + \lambda_n$ is a $(0,r)$-partition 
of $m$ then $(e+1)\lambda_0 + (e+1)\lambda_1 + \cdots + (e+1)\lambda_n$ is an $(e,r)$-partition of $(e+1)m$ 
-- this partially explains the ``doubling of parts'' in solving Tanton's variant of Bachet's problem. 
The partition $m = 
1+1+ \cdots +1+1
$ 
is always an $(e,r)$-partition of $m$ 
and every $(e,r)$-partition of $m$ is both an $(e+1,r)$-partition and an $(e,r+1)$-partition of $m$. 

The $r$-complete partitions were first classified and enumerated for every integer $m$ by Park \cite{Par} 
and the minimal ones by O'Shea \cite{Osh} (a partial enumeration for the $r=1$ case) and by R{\o}dseth 
in two papers: \cite{RodMpart} for the $r=1$ case and \cite{Rod} for $r \geq 2$. This present piece will 
classify and enumerate the minimal $(e,r)$-partitions for all $e$ and $r$ and, as one would expect, our results 
agree with those of Park and R{\o}dseth when the error term $e$ is set equal to $0$. First, we can classify 
the family of minimal $(e,r)$ partitions as a collection of lattice points in polyhedra.

\vspace{.3cm}

\noindent {\bf Theorem~\ref{the:iff} and Proposition~\ref{pro:min}} 
The $(e,r)$-partitions with $n+1$ parts are precisely the positive integer points 
$(\lambda_0, \lambda_1, \ldots, \lambda_n)$ that satisfy the linear inequalities 
$\lambda_0 \leq \lambda_1 \leq \cdots \leq \lambda_n$, the inequalities 
$\textup{ineq}_i \, : \, \lambda_i \leq (e+1) + r \sum_{j=0}^{i-1}{\lambda_j}$ for each $i \leq n$. 
Furthermore, the partitions with $1$-norm 
$|\lambda| := \lambda_0 + \lambda_1 + \cdots + \lambda_n$ lying in the interval 
$[(e+1)\frac{(r+1)^{n}-1}{r}+1, (e+1)\frac{(r+1)^{n+1}-1}{r}]$ are the minimal such partitions. 

\vspace{.3cm}

We enumerate the minimal $(e,r)$-partitions using {\em Brion's theorem} \cite{Bri}, a general 
formula for enumerating lattice points in polyhedra. 
Schematically Brion's formula is as follows: 
given a simple polyhedron we can form at each vertex 
${\bf v}$ a {\em vertex cone} $\KK_{\bf v}$ and the generating function 
$\sigma_{\KK_{\bf v}}$ for the lattice points in that vertex cone. 
Brion's formula states that the lattice points in the polyhedron 
are encoded precisely by 
$\sum_{\bf v} \sigma_{\KK_{\bf v}}$ where the sum is taken over all 
vertices of the said polyhedron. Section~\ref{sec:brion} will provide the necessary background for 
Brion's Theorem and makes this article mostly self contained.  

This approach of using discrete geometric tools to enumerate families of integer partitions
described by linear inequalities contributes to an existing literature that was especially inspired by the 
1997 papers \cite{BME1, BME2} of Bousquet-M{\'e}lou and Eriksson on {\em lecture hall partitions}.  
These partitions, having originally arose in the context of Bott's formula for enumerating the lengths of 
reduced words in the affine Coxeter groups of type $C$, can be described in terms of lattice points 
in polyhedra. They are celebrated not only because they provide a rich connection between reflection 
groups, discrete geometry and integer partitions but because they also provide a finite version of 
Euler's classical odd-distinct theorem.  A few papers of many that the reader might be interested in are 
Pak's use of lattice point enumeration \cite{Pak} when the family of integer partitions are 
points in a unimodular cone and Corteel, Lee and Savage's \cite{savage-five} practical and 
philosophical introduction to discrete geometric methods in integer partitions. 
Recently, perhaps in homage to the Coxeter context from which lecture hall partitions arose,  Beck, Bliem, Braun and 
Savage \cite{BBBS} study the actions of finite reflections on cones of symmetric compositions. While the relaxed 
complete partitions bear no direct relation to lecture hall partitions they are another example of a 
generalization of a classical problem from integer partitions that is significantly enhanced 
and enriched by viewing it through a modern geometric lens. 

In Section~\ref{sec:enumerate} we enumerate the minimal complete partitions. 
We begin by replacing the original $\lambda$-inequalities of Theorem~\ref{the:iff} and 
Proposition~\ref{pro:min} with $\mu$-inequalities by using the linear transformations 
$
\mu_i := (e+1)(r+1)^i - \lambda_i.
$ 
This translates the interval condition on $|\lambda|$ to all vectors $\mu$ with coordinate sum 
satisfying $0 \leq |\mu| < (e+1)(r+1)^n$, so 
we must only find a generating function for the enumeration of the $\mu$ vectors up to 
order $(e+1)(r+1)^n$. Proposition~\ref{pro:type} shows that the set of $\mu$-inequalities 
describes an $(n+1)$-dimensional simple polyhedron $P_n(e,r)$, combinatorially (but not metrically) 
isomorphic to the $(n+1)$-cube for every $e$ and $r$, and with vertices $\{ \mu(\Set) \}$ that are fully 
determined by the subsets $\Set \subseteq \{0,1,\ldots,n\}$. 

 \begin{figure}[h] 
\vspace{.2cm}
\centering 
\def\svgwidth{500pt} 
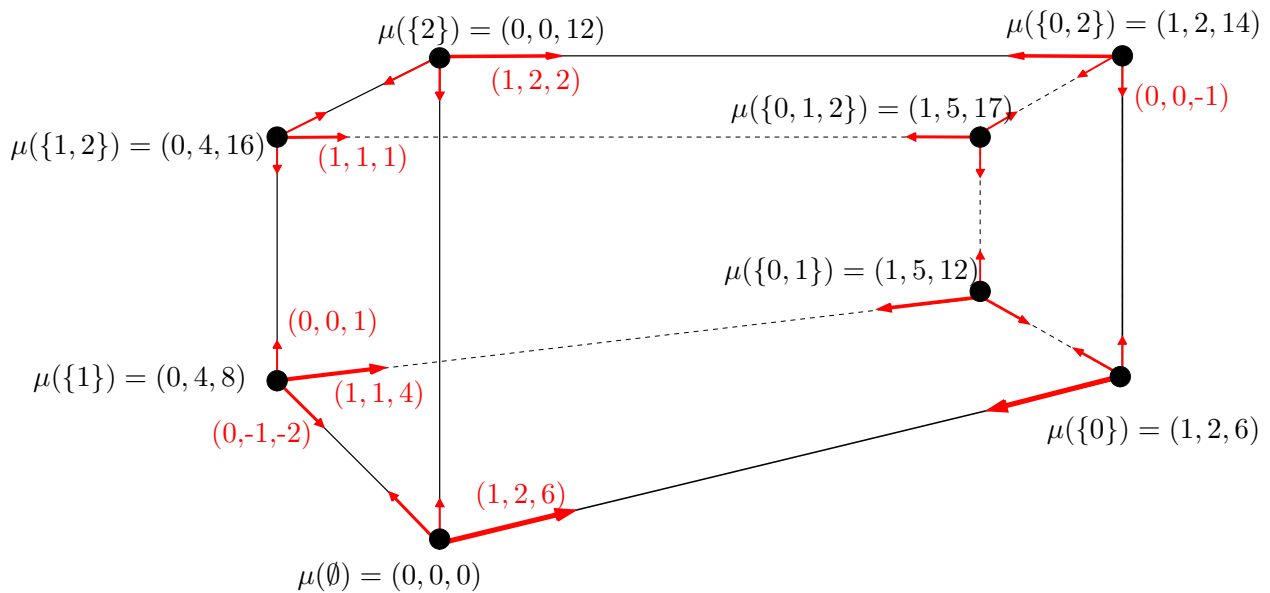 \label{fig:poly}
\caption{The polyhedron $P_2(1,2)$. Vertices and some rays of vertex cones are labelled.}
 \end{figure}

We can then apply Brion's formula 
as it is stated earlier to the $2^{n+1}$ vertex cones of $P_n(e,r)$ but need only account for 
those vertex cones whose order is less than 
$(e+1)(r+1)^n$. We claim that almost all vertex cones are of too high an order: for $r \geq 2$ we must 
only account for $n+1$ vertex cones -- the vertex cone at the origin and all but one of the vertex cones of 
those vertices neighboring the origin in $P_n(e,r)$. If $r=1$ we must also account for an additional 
${n-1}\choose{2}$ neighbors of those neighbors of the origin and no more. This yields our main result. 

\vspace{.3cm}

\noindent {\bf Theorems~\ref{thm:r>1} and \ref{thm:r=1} } The minimal relaxed complete $(e,r)$-partitions 
can be enumerated by the generating function 
(where $\chi(r) = 1$ when $r=1$ and $0$ for $r\geq 2$) 
$$
\displaystyle \sigma_{P_n(e,r)}(x) 
  = 
  \sigma_{\KK_{\mu(\emptyset)}}(x) 
  + 
  \sum_{j=1}^{n} \sigma_{\KK_{\mu(\{j\})}}(x) 
  + 
   \chi(r) \cdot \sum_{j_1=1}^n \sum_{j_2 = j_1+2}^n \sigma_{\KK_{\mu(\{j_1, j_2\})}}(x) + O(x^{(e+1)(r+1)^n}).
$$

Specializing the results for the minimal $(1,2)$-partitions with $n+1 = 3$ parts, we see that 
$12$ lies in the interval $[9, 26]$. There are six minimal $(1,2)$-partitions of $12$: 
$1+3+8$, $1+4+7$, $2+2+8$, $2+3+7$, $2+4+6$ and $2+5+5$. The enumerating function, 
consulting Figure~\ref{fig:poly}, reflects this by looking at the coefficient of 
$x^{26-12} = x^{14}$ in the generating function 
\begin{align*} 
\sigma_{P_2(1,2)}(x) \, = \, &  
x^{|(0,0,0)|}
\frac{1}{1-x^{|\red{(1,2,6)}|}}
\frac{1}{1-x^{|\red{(0,1,2)}|}}
\frac{1}{1-x^{|\red{(0,0,1)}|}} \\ 
 + &
x^{|(0,4,8)|}
\frac{1}{1-x^{|\red{(0, {-1},-2)}|}}
\frac{1}{1-x^{|\red{(1,1,4)}|}}
\frac{1}{1-x^{|\red{(0,0,1)}|}} \\ 
 + &
x^{|(0,0,12)|}
\frac{1}{1-x^{|\red{(0,0,-1)}|}}
\frac{1}{1-x^{|\red{(1,2,2)}|}}
\frac{1}{1-x^{|\red{(0,1,1)}|}}  + O(x^{18})
\end{align*}
$
\displaystyle  
= 
\frac{1}{(1-x^9)(1-x^3)(1-x^1)} - \frac{x^{15}}{(1-x^3)(1-x^6)(1-x^1)} - \frac{x^{13}}{(1-x)(1-x^5)(1-x^2)} + O(x^{18}).
$

\noindent 
Note that the coefficient of $x^{14}$ in $\sigma_{P_2(1,2)}(x)$ is $(7-1)=6$. Full justification for the derivation 
of the polyhedron $P_n(e,r)$ and its lattice point enumerating function $\sigma_{P_n(e,r)}(x)$ will be given 
in Sections~\ref{sec:brion} and~\ref{sec:enumerate}. 

\section{Classifying relaxed complete partitions} \label{sec:classify}

We begin with classifying the relaxed complete partitions as a collection of lattice points in polyhedra 
and describe the minimal such partitions amongst these. A novel feature of the classification is 
that the defining hyperplanes for the polyhedra that cut out the $(e,r)$-partitions arise from 
simply translating the defining hyperplanes of the $(0,r)$-partitions by a linear factor 
of $e$. In this section we call $\{ \sum_{j=0}^{n} \alpha_j \lambda_j : \alpha_j \in \{ 0,1,\ldots,r \} \}$ 
the {\em r-cover} of $\lambda_0+\lambda_1+\cdots+\lambda_n$. 
The following theorem the proof of which closely follows that of \cite{Par}. 

\begin{theorem} \label{the:iff} 
Let $m = \lambda_0 + \lambda_1 + \cdots + \lambda_n$ be a partition with $\lambda_0 \leq e + 1$. 
Then $m = \lambda_0 + \lambda_1 + \cdots + \lambda_n$ is an $(e,r)$-partition if and only if 
$\lambda_i \leq (e+1) + r \sum_{j=0}^{i-1}{\lambda_j}$, for all $i \leq n$.
\end{theorem}
\begin{proof} 
For sufficiency observe that if $\lambda_i > (e+1) + r\sum_{j=0}^{i-1}{\lambda_j}$ 
then the shifted set $r\sum_{j=0}^{i-1}{\lambda_j} + \{1, 2, \ldots, e+1\}$ would also 
be omitted from the $r$-cover of $\lambda_0 + \lambda_1 + \cdots + \lambda_n$. 

We show necessity by inducting on the number of parts in the partition. 
If $n=0$ (the number of parts equals $1$) then $\lambda_0 = m \leq e+1$, in accordance with our hypothesis. 
Let $\lambda_0 + \lambda_1 + \cdots + \lambda_{n-1}$ be an $(e,r)$-partition onto which we append any part 
$\lambda_n$, with $\lambda_{n-1} \leq \lambda_n \leq (e+1) + r \sum_{j=0}^{n-1}{\lambda_j}$. 
We wish to show that every positive integer $l \leq r(\lambda_0 + \lambda_1 + \cdots + \lambda_n)$ is 
within a distance of no greater than $e+1$ of some integer in the $r$-cover of 
$\lambda_0 + \lambda_1 + \cdots + \lambda_n$

If $l \leq r(\lambda_0 + \lambda_1 + \cdots + \lambda_{n-1})$ 
then by our inductive hypothesis we have nothing to show. 
So we can fix $l \leq r \sum_{j=0}^{n}{\lambda_j}$ and $l > r \sum_{j=0}^{n-1}{\lambda_j}$. 
But in this case there will always exist an $1 \leq \alpha_n \leq r$ such that 
$
(\alpha_n - 1)\lambda_n + r \sum_{j=0}^{n-1}{\lambda_j} 
< l  \leq 
\alpha_n \lambda_n + r \sum_{j=0}^{n-1}{\lambda_j}
$ 
or $\alpha_n = \left\lceil \frac{l - r \sum_{j=0}^{n-1}{\lambda_j} }{\lambda_n} \right\rceil$. 
Again, since $l - \alpha_n \lambda_n \leq r \sum_{j=0}^{n-1}{\lambda_j}$ our inductive hypothesis 
tells us that 
$l - \alpha_n \lambda_n$ is within distance $e+1$ of an integer in the $r$-cover of 
$\lambda_0 + \lambda_1 + \cdots + \lambda_{n-1}$ 
and so $l$ must be within distance $e+1$ of an integer 
in the $r$-cover of 
$\lambda_0 + \lambda_1 + \cdots + \lambda_n$.
\end{proof}

Next, we classify the {\em minimal} $(e,r)$-partitions. We have already observed that 
$\lambda_0 \leq (e+1)$ and we can in turn note that 
$\lambda_1 \leq (e+1) + r \lambda_0 \leq (e+1)(r+1)$. An inductive argument shows 
\begin{equation} \label{eqn:upper}
\lambda_i \leq (e+1) (r+1)^i
\end{equation}
for all $(e,r)$-partitions $\lambda_0 + \lambda_1 + \cdots + \lambda_n$. 
Hence if $m = \lambda_0 + \lambda_1 + \cdots + \lambda_n$ is an $(e,r)$-partition then 
the sum of the parts in the partition cannot exceed $\sum_{i=0}^n{(e+1)(r+1)^i} = \frac{(e+1)}{r}((r+1)^{n+1}-1)$. 
That is, 
$$
m \leq \frac{e+1}{r}((r+1)^{n+1}-1) < \frac{e+1}{r}(r+1)^{n+1} 
\, \, \, \, \textup{or} \, \, \, \,   
\textup{log}_{r+1}\left(\frac{rm}{e+1}\right) < n+1.
$$ 
Since $n+1$ is an integer then the integer part of $\textup{log}_{r+1}(\frac{rm}{e+1})$ 
is strictly less than $n+1$
i.e.  $\lfloor \textup{log}_{r+1}(\frac{rm}{e+1}) \rfloor \leq n$. 
This tells us that an $(e,r)$-partition of $m$ must have 
{\em at least} $\lfloor \textup{log}_{r+1}(\frac{rm}{e+1}) \rfloor + 1$ parts. This 
number will suffice.

\begin{proposition} \label{pro:min}
A minimal $(e,r)$-partition of $m$ has exactly $n+1 = \lfloor \textup{log}_{r+1}(\frac{rm}{e+1}) \rfloor + 1$ parts.
\end{proposition}

\begin{proof}
It suffices to show that for any integer $m$ with 
$n:= \lfloor \textup{log}_{r+1}(\frac{rm}{e+1}) \rfloor$ 
the partition whose parts are exactly those in the multiset 
$$
\{\, (e+1), \,(e+1)(r+1), \, (e+1)(r+1)^2, \ldots, (e+1)(r+1)^{n-1}, \,  m - \frac{(e+1)}{r}((r+1)^n - 1) \}
$$
is an $(e,r)$-partition of $m$.

To see this first observe that 
$1 + (r+1) + (r+1)^2 + \cdots + (r+1)^{n-1}$ is a $(0,r)$-partition of $\frac{1}{r}((r+1)^n - 1)$ 
and so $(e+1) + (e+1)(r+1) + (e+1)(r+1)^2 + \cdots + (e+1)(r+1)^{n-1}$ is an $(e,r)$-partition of 
$\frac{(e+1)}{r}((r+1)^n - 1)$. Next, for each $0 \leq \alpha \leq r$ the shifted set 
$$
\alpha \cdot (m - \frac{(e+1)}{r}((r+1)^n - 1)) 
\, + \,
\left\lbrace \sum_{i=0}^{n-1} \alpha_i (e+1)(r+1)^i : \alpha_i \in \{ 0,1,\ldots,r \} \right\rbrace
$$ 
will not omit any consecutive $(e+1)$ integers and, since 
$m < \frac{e+1}{r}(r+1)^{n+1}$, the union over $1 \leq \alpha \leq r$ of these shifted sets 
not only range from $1$ through to $rm$ but do not omit any $(e+1)$ consecutive integers. 
This union, of course, is precisely the $r$-cover of the partition whose parts consist of the 
elements from the multiset. 
\end{proof}

Recall that if $\lambda_0 + \lambda_1 + \cdots + \lambda_n$ is a $(0,r)$-partition of $m$ then 
$(e+1)\lambda_0 + (e+1)\lambda_1 + \cdots + (e+1)\lambda_n$ is a $(e,r)$-partition of $(e+1)m$. 
The above proposition tells us that  more is true: {\em minimality} is also preserved by the 
multiplication of parts by $e+1$. This explains the ``doubling of parts'' in solving the Tanton 
variant of Bachet's problem with $80=2+6+18+54$ from the solution of 
$40= 1+3+9+27$ to the original Bachet problem. 

The classification of minimal $(e,r)$-partitions closely mirrored the development of 
complete partitions by Park \cite{Par} but at this point we depart from the 
analysis of \cite{Par} and R{\o}dseth \cite{RodMpart, Rod} who did not view 
the minimal complete partitions as lattice points in polyhedra, instead using nice 
inductive arguments, correctly 
anticipating what the enumerating function should be and proceeding to prove that 
this was the case.  In the remainder of this article, our treatment here will differ 
significantly, focusing instead on the geometric character of the relaxed complete 
partitions.

\section{A primer on Brion's Theorem} \label{sec:brion}

The remarkable formula of Brion \cite{Bri} states that 
the lattice points in $\PP$ are encoded precisely by the monomials appearing 
in the sum of generating functions for the vertex cones. In Section~\ref{sec:enumerate} 
we will enumerate the minimal relaxed complete partitions using Brion's formula. In order 
to make this article self contained we will give a brief description here of Brion's 
formula, in a restrictive manner that is best suited to our needs. 

Given an integer vector ${\bf u}$ define the {\em primitive part of ${\bf u}$} 
as $\textup{p}({\bf u}) := \frac{1}{\textup{gcd}({\bf u})} {\bf u}$, where 
the $\textup{gcd}({\bf u})$ is the greatest common divisor of the entries of ${\bf u}$.  
Call ${\bf u}$ {\em primitive} if $\textup{gcd}({\bf u})$ equals $1$. 
Given two integer vectors ${\bf u}$ and ${\bf u}^\prime$ we will 
denote $\textup{p}({\bf u}^\prime - {\bf u})$ by 
$\textup{p}_{\bf u}({\bf u}^\prime)$, 
as in the {\em primitive vector of ${\bf u}^\prime$ relative to ${\bf u}$}. 
Note that $\textup{p}_{\bf u}({\bf u}^\prime) = - \textup{p}_{{\bf u}^\prime}({\bf u})$
Given a vertex ${\bf v}$ of a bounded integer polyhedron $\PP$, the 
{\em vertex cone of $\PP$ at ${\bf v}$} is the smallest cone with 
apex ${\bf v}$ that contains $\PP$. 
Letting $\NN({\bf v})$ denote the neighbors of ${\bf v}$ in $\PP$, we can 
write the vertex cone at ${\bf v}$ as 
$$
\KK_{\bf v} := {\bf v} + \R_{\geq 0} \{ \textup{p}_{\bf v}({\bf v}^\prime) \,
: \, {\bf v}^\prime \in \NN({\bf v})\} 
$$
where the latter term is understood as the non-negative real span of the set of
primitive vectors 
of the neighbors of ${\bf v}$ relative to ${\bf v}$ itself.  
The polyhedron $\PP$ is said to {\em simple} if the generators 
$
\{ \textup{p}_{\bf v}({\bf v}^\prime) \,
: \, {\bf v}^\prime \in \NN({\bf v})\}
$ 
of each vertex cone form a linearly independent set. 
Assuming that the polyhedron $\PP$ is full dimensional we say that a vertex cone 
$\KK_{\bf v}$ is {\em unimodular} if the square matrix whose rows equal 
$\{ \textup{p}_{\bf v}({\bf v}^\prime) \, : \, {\bf v}^\prime \in \NN({\bf
v})\}$ has determinant $\pm 1$. 
This means that every integer point that lies in the vertex cone $\KK_{\bf v}$ 
can be written uniquely as the apex plus a non-negative 
integer combination of the primitive vectors that generate the cone. In this
case the set 
of integer points in a vertex cone $\KK_{\bf v}$ of a simple polyhedron $\PP
\subseteq \R^{n+1}$ 
can be encoded by a sum of monomials which can be 
written as a generating function (where ${\bf z}^{\bf u} := z_0^{u_0}z_1^{u_1}
\cdots z_n^{u_n}$): 
$$
\sigma_{\KK_{\bf v}}({\bf z}) 
\, := \, 
\sum_{{\bf m} \in \KK_{\bf v}\cap \N^{n+1}} {\bf z}^{\bf m}
\, = \, 
{\bf z}^{\bf v} \prod_{{\bf v}^\prime \in \NN({\bf v})} \frac{1}{1 - {\bf
z}^{\textup{p}_{\bf v}({\bf v}^\prime)}}.
$$
The remarkable formula of Brion~\cite{Bri}, specialized here for the
bounded, integer and simple polyhedron $\PP$ with all its 
vertex cones being unimodular, states that the lattice points in $\PP$ are encoded
precisely by the monomials appearing 
in the sum of generating functions for the vertex cones:
\begin{equation} \label{eqn:BrionMulti}
\sigma_{\PP}({\bf z}) 
\, = \, 
\sum_{{\bf v} \, \textup{a vertex of} \, \PP} 
\sigma_{\KK_{\bf v}}({\bf z})
\, = \, 
\sum_{{\bf v} \, \textup{a vertex of} \, \PP} 
{\bf z}^{\bf v} \prod_{{\bf v}^\prime \in \NN({\bf v})} \frac{1}{1 - {\bf
z}^{\textup{p}_{\bf v}({\bf v}^\prime)}}.
\end{equation}
Brion's formula holds too when the vertex cones of the simple polyhedron are not
unimodular but this requires of a more 
detailed description which we do not need to call upon here. 
Note that to enumerate all lattice points with common
$1$-norm (denoted by $|\cdot|$) it suffices to
set the variables $z_0,z_1,\ldots,z_n$ in (\ref{eqn:BrionMulti}) to a common
variable $x$ to yield: 
\begin{equation} \label{eqn:Brion}
\sigma_{\PP}(x) 
\, = \, 
\sum_{\bf v} 
\sigma_{\KK_{\bf v}}(x) 
\, = \, 
\sum_{\bf v} 
x^{|{\bf v}|} \prod_{{\bf v}^\prime \in \NN({\bf v})} 
\frac{1}{1 - x^{|{\textup{p}_{\bf v}({\bf v}^\prime)}|} }.
\end{equation}
where the sum is taken over all ${\bf v}$ that are vertices of $\PP$. 

For a given fixed ${\bf v}$ if $|{\textup{p}_{\bf v}({\bf v}^\prime)}| < 0$ then 
$
\displaystyle 
\frac{1}{1 - x^{|{\textup{p}_{\bf v}({\bf v}^\prime)}|} } 
\, = \, 
\frac{- x^{|{\textup{p}_{\bf v^\prime}({\bf v})}|}}{1 - x^{|{\textup{p}_{\bf v^\prime}({\bf v})}|} } 
$ 
and so this means that the vertex cone 
\begin{equation} \label{eqn:Flip}
\sigma_{\KK_{\bf v}}(x) 
\, = \, 
x^{|{\bf v}|} 
\prod_{{\bf v}^\prime \, : \, |{\textup{p}_{\bf v}({\bf v}^\prime)}| < 0 }
-x^{|{\textup{p}_{\bf v^\prime}({\bf v})}|}
\prod_{{\bf v}^\prime \, : \, |{\textup{p}_{\bf v}({\bf v}^\prime)}| < 0} 
\frac{1}{1 - x^{|\textup{p}_{{\bf v}^\prime}({\bf v})|}} 
\prod_{{\bf v}^\prime \, : \, |{\textup{p}_{\bf v}({\bf v}^\prime)}| > 0} 
\frac{1}{1 - x^{|\textup{p}_{\bf v}({\bf v}^\prime)|}}.
\end{equation}
where each product is taken over ${\bf v}^\prime \in \NN({\bf v})$. 
Written in this form, every monomial in the vertex cone's generating 
function $\sigma_{\KK_{\bf v}}(x)$ has degree at least 
$
O(\KK_{\bf v}) : = 
|{\bf v}| + 
\sum_{ \substack{{\bf v}^\prime \, : \,  |{\textup{p}_{\bf v}({\bf v}^\prime)}| < 0 }}  
|{\textup{p}_{\bf v^\prime}({\bf v})}|
$ 
and we call this quantity the {\em order} of the vertex cone. Permitting an abuse of notation, 
we will sometimes write $O({\bf v})$ for the order of the vertex cone whose apex is ${\bf v}$ in the 
polyhedron $\PP$.

\section{Enumerating minimal relaxed complete partitions} \label{sec:enumerate}

An alternative classification of the minimal relaxed complete partitions are as follows: 
the $(e,r)$-partitions 
with $n+1$ parts are precisely the positive integer points 
$\lambda = (\lambda_0, \lambda_1, \ldots, \lambda_n)$ that satisfy the linear inequalities 
$\lambda_0 \leq \lambda_1 \leq \cdots \leq \lambda_n$, the inequalities 
$\textup{ineq}_i \, : \, \lambda_i \leq (e+1) + r \sum_{j=0}^{i-1}{\lambda_j}$ for each $i \leq n$. 
Furthermore, the partitions with $1$-norm 
$|\lambda| := \lambda_0 + \lambda_1 + \cdots + \lambda_n$ lying in the interval 
$ [(e+1)\frac{(r+1)^{n}-1}{r}+1, (e+1)\frac{(r+1)^{n+1}-1}{r}]$ are the minimal such partitions. 

Rather than computing the enumerating function of the minimal $(e,r)$-partitions 
directly, we will instead describe another collection of 
lattice points with an enumerating function that is equivalent to that of the minimal $(e,r)$-partitions. 
Set 
$$
\mu_i := (e+1)(r+1)^i - \lambda_i 
$$ 
for each $i=0,1,\ldots,n$. Note that $|\mu| = \frac{e+1}{r}((r+1)^{n+1}-1) - |\lambda|$ 
and that the constraints on the $\lambda_i$'s from Theorem~\ref{the:iff} translate to  
$$
0 \leq \mu_0 \leq e 
\hspace{.5in} \textup{and} \hspace{.5in} 
r \sum_{j = 0}^{i-1} \mu_j \leq \mu_i \leq (e+1)(r)(r+1)^{i-1} + \mu_{i-1}.
$$
Hence we have a one-to-one correspondence between the $(e,r)$-partitions 
$\lambda = (\lambda_0, \ldots, \lambda_n)$ with 
$ |\lambda| \in [(e+1)\frac{(r+1)^{n}-1}{r}+1, (e+1)\frac{(r+1)^{n+1}-1}{r}]$
and the non-negative integer vectors $\mu = (\mu_0, \ldots,\mu_n)$ with 
$|\mu| \in [0,(e+1)(r+1)^n-1]$ that satisfy the above inequalities. Furthermore, 
since $|\mu| = \frac{e+1}{r}((r+1)^{n+1}-1) - |\lambda|$ the number of 
minimal $(e,r)$-partitions of $|\lambda|$ equals the number of $\mu$ vectors 
with $n+1$ parts with $|\mu| = \frac{e+1}{r}((r+1)^{n+1}-1) - |\lambda|$. 
In summary, it will suffice to find a generating function for the norms of the transformed 
$\mu$ vectors, $\sum x^{|\mu|} + O(x^{(e+1)(r+1)^n})$ 
where the sum is taken over all vectors $\mu$ of length $n+1$. 

In light of the above  we wish 
to study the lattice points $\mu = (\mu_0, \mu_1, \ldots, \mu_n)$ that 
for every $i=1,2,\ldots,n$ live in the halfspaces 
$$
\textup{min}^+(i) \, := \, 
\{ \mu \in \R^{n+1} \, : \, 
r \sum_{j = 0}^{i-1} \mu_j \leq \mu_i
\}
\hspace{.15in} \textup{and} \hspace{.15in} 
\textup{max}^+(i) \, := \, 
\{ \mu \in \R^{n+1} \, : \, 
\mu_i \leq (e+1)(r)(r+1)^{i-1} + \mu_{i-1}
\}, 
$$
in addition to 
$
\textup{min}^+(0) \, := \, 
\{ \mu \in \R^{n+1} \, : \, 0 \leq \mu_0 \}
$
and 
$
\textup{max}^+(0) \, := \, 
\{ \mu \in \R^{n+1} \, : \, \mu_0 \leq e \}.
$
For every $i = 0,1,\ldots,n$ let $\textup{min}(i)$ and $\textup{max}(i)$ denote
the 
defining hyperplanes of the above, that is with every occurrence of
$\leq$ replaced 
by equality in $\textup{min}^+(i)$ and $\textup{max}^+(i)$ respectively. 

We define the polyhedron $P_n(e,r)$ as the common intersection of the above
$\mu$ halfspaces. We claim that $P_n(e,r)$ is a bounded integer polyhedron that is simple 
and whose vertex cones are unimodular and so Brion's formula as stated in 
Section~\ref{sec:brion} can be applied to enumerating its lattice points. We will then 
explicitly calculate the vertices of $P_n(e,r)$ and the primitive vectors that generate each 
of $P_n(e,r)$'s vertex cones. 

\begin{remark}
Note that when $e=0$ we have $\textup{min}^+(0) \cap \textup{max}^+(0)$ equals
the hyperplane defined by $\mu_0 = 0$. 
In the next proposition (and in some later results) we implicitly assume that $e \geq 1$; 
for $e=0$ the statement is exactly the same except that the polyhedron is of dimension 
$n$ instead of $n+1$. Similar minor omissions apply to later claims. 
\end{remark}

Given a set $\Set \subseteq \{ 0,1,2,\ldots,n \}$ define the set of vectors 
$
\mu(\Set) \, := \, 
\{ \mu \in \textup{max}(i) \, : \, i \in \Set \} 
\, \cap \, 
\{ \mu \in \textup{min}(i) \, : \, i \notin \Set \}. 
$
These sets consist of single elements and are precisely the vertices of $P_n(e,r)$.

\begin{proposition} \label{pro:type}
The polyhedron $P_n(e,r)$ is an $(n+1)$-dimensional simple polyhedron. 
The vertices of $P_n(e,r)$ are precisely $\{ \mu(\Set) \, : \, \Set \subseteq \{
0,1,2,\ldots,n \} \}$. 
Two vertices $\mu(\Set)$ and $\mu(\TT)$ share an edge if and only if their
symmetric difference $\Set \triangle \TT$ is a singleton. 
\end{proposition}

\begin{proof}
It is straightforward to check that for any $i=0,1,2,\ldots,n$ there is no point 
$\mu \in P_n(e,r)$ that simultaneously satisfies both $\textup{max}(i)$ and
$\textup{min}(i)$.
Next, the normals of the hyperplanes that define $\mu(\Set)$ are a set of
linearly independent 
vectors since each $\textup{min/max}(i)$ sequentially introduces a new
independent variable $\mu_i$. 
Thus each $\mu(\Set)$ equals the 
intersection of $n+1$ affinely independent hyperplanes in $\R^{n+1}$. 
While no other defining hyperplane of $P_n(e,r)$ can contain $\mu(\Set)$, it is
nonetheless contained 
in the interior of the other defining halfspaces of $P_n(e,r)$. Thus, for every 
$\Set \subseteq \{ 0,1,2,\ldots,n \}$, $\mu(\Set)$ must be a vertex of
$P_n(e,r)$ and 
furthermore it is defined by $n+1$ halfspaces and equivalently must have $n+1$
edges. 
That is,  $P_n(e,r)$ is an $(n+1)$-dimensional simple and bounded polyhedron. 

As for the edges, $|\Set \triangle \TT| = 1$ if and only if $\mu(\Set)$ and 
$\mu(\TT)$ have precisely $n$ out of each of their $n+1$ defining hyperplanes 
in common and, since $P_n(e,r)$ is simple, this is 
equivalent to $\mu(S)$ and $\mu(\TT)$ being neighbors in $P_n(e,r)$. 
\end{proof}

In order to apply Brion's formula to $P_n(e,r)$ we also need that the polyhedron $P_n(e,r)$ 
have vertex cones that are all unimodular. 

\begin{lemma} \label{lem:unimodCones}
The vertex cones of $P_n(e,r)$ are unimodular.
\end{lemma}
\begin{proof} 
By Proposition~\ref{pro:type} we need only consider pairs of sets $\Set$ and
$\TT$ 
where $\Set \triangle \TT = \{ j \}$. For now let us assume that  $\TT = \Set
\cup \{ j \}$
and we will show that 
${\bf p} := \textup{p}_{\mu(\Set)}(\mu(\TT))$ equals $0$ in entries $0$
through $j-1$ 
and $1$ in entry $j$. Since $\textup{p}_{\mu(\TT)}(\mu(\Set)) = -
\textup{p}_{\mu(\Set)}(\mu(\TT))$ 
the case of $\Set = \TT \cup \{ j \}$ follows in exactly the same fashion but
with $-1$ in entry $j$. 
Consequently, the primitive vectors of the vertex cone of any given vertex 
of $P_n(e,r)$ form the rows of a square matrix with determinant $\pm 1$. 

For every $i < j $ we have that $\mu(\Set)_j = \mu(\TT)_j$ since the sets $\Set$
and $\TT$ 
agree when restricted to the support of $\{ 0,1,2,\ldots, j-1\}$. Hence 
$\textup{p}_{\mu(\Set)}(\mu(\TT))_i = 0$ for $i < j$. The $j$-th entry of the
difference 
vector is  
$
\mu(\TT)_j - \mu(\Set)_j  = r(e+1)(r+1)^{j-1} + \mu(\TT)_{j-1} - r
(\sum_{k=0}^{j-1}{\mu(\Set)_k}) 
$ which is a positive integer -- this difference must be positive since it is 
exactly $r(\lambda_0+\cdots+\lambda_{j-1}) - \lambda_{j-1}$ in terms of the original 
$\lambda$-formulation of the $(e,r)$-partition. 

For the $i = j+1,\ldots ,n$, recalling that $\Set$ and $\TT$ agree outside of
$\{ j \}$, the 
$i$-th entry of the difference equals 
\begin{equation} \label{eqn:diff}
\mu(\TT)_i - \mu(\Set)_i
= \left\lbrace 
\begin{array}{lcc} 
\mu(\TT)_{i-1} - \mu(\Set)_{i-1} & : & i \in \TT \\
r(\sum_{k=j}^{i-1}({\mu(\TT)_ k - \mu(\Set)_k})) & : & i \notin \TT
\end{array} \right.
\end{equation}

Either way, each $\mu(\TT)_{i} - \mu(\Set)_{i}$ is recursively defined from 
$\{ \mu(\TT)_ k - \mu(\Set)_k : k=j,\ldots, i-1 \}$ 
and each such difference is a multiple of $\mu(\TT)_{j} - \mu(\Set)_{j}$. 
Hence, $\mu(\TT)_{j} - \mu(\Set)_{j}$ equals 
$\textup{gcd}(\mu(\TT) - \mu(\Set))$ 
and so $\textup{p}_{\mu(\Set)}(\mu(\TT))_j = 1$. 
\end{proof}

Applying Brion's formula from Equation~(\ref{eqn:Flip}) requires explicit calculation of 
the $1$-norms of the vertices of $P_n(e,r)$ and their vertex cones. Since we are only interested 
in the lattice points of $P_n(e,r)$ that have $1$-norm less than $(e+1)(r+1)^n$ then we can 
eliminate all vertex cones that have order greater than or equal to $(e+1)(r+1)^n$.

We begin by calculating $|\mu(\Set)|$ for each $\Set \subseteq \{ 0,1,2,\ldots,n \}$. 
We write $\mu$ as shorthand for $\mu(\Set)$. If $k <  k' - 1$ then we write $|\mu|_k^{k'-1}$ for
$\mu_k + \mu_{k+1} + \cdots + \mu_{k' - 1}$. We denote $|\mu|_0^{k'-1}$ as simply 
$|\mu|^{k'-1}$, $|\mu|_k^{n}$ as simply $|\mu|_k$ and, as before, $|\mu|_0^{n}$ as simply $|\mu|$. 

\begin{lemma} \label{lem:runs}
\noindent \textup{\bf (a)} If $k, \ldots, k'-1 \notin \Set$ 
then 
$$ 
\displaystyle |\mu|_k^{k'-1} =  ((r+1)^{k' - k} -1) |\mu|^{k-1} 
\hspace{.2in} 
\textup{and} 
\hspace{.2in} 
|\mu|^{k'-1} =  (r+1)^{k' - k} |\mu|^{k-1}.
$$


\noindent \textup{\bf (b)} If $k \notin \Set$ then $\displaystyle \mu_k = \frac{r}{r+1}|\mu|^k$. 

\noindent \textup{\bf (c)} If $0 < k, \ldots, k'-1 \in \Set$ then 
$
\displaystyle |\mu|_k^{k'-1} =  
\frac{(e+1)(r+1)^k}{r} [(r+1)^{k' - k} - 1 - (k' - k) \frac{r}{r+1}] + (k' - k) \mu_{k-1} .
$

\noindent \textup{\bf (d)} If $0, \ldots, k'-1 \in \Set$ then 
$
\displaystyle |\mu|_0^{k'-1} =  
\frac{(e+1)}{r}[(r+1)^{k'} - 1 - k' \frac{r}{e+1}].
$
\end{lemma}
\begin{proof}
{\bf (a)} Since $k \notin \Set$ then $\mu_k = r |\mu|^{k-1}$ and, since $k+1
\notin \Set$ then 
$\mu_{k+1} = r |\mu|^{k} = r (|\mu|^{k-1} + \mu_k) = r (|\mu|^{k-1} + r
|\mu|^{k-1}) = r (r+1) |\mu|^{k-1}$. 
Recursively we have $\mu_{k+i} = r (r+1)^i |\mu|^{k-1}$, hence 
$$
|\mu|_k^{k'-1} = r \sum_{i = 0}^{k'-1-k}(r+1)^i |\mu|^{k-1} = ((r+1)^{k' - k} - 1) |\mu|^{k-1} 
$$
and, by definition, this implies that $|\mu|^{k'-1} = |\mu|_k^{k'-1} + |\mu|^{k-1}$.

\noindent {\bf (b)} Since $|\mu|^j = \mu_j + |\mu|^{j-1}$ and $j \notin \Set$ 
then $\frac{1}{r} \mu_j = |\mu|^{j-1}$ implying $|\mu|^j = (1 + \frac{1}{r})\mu_j$. 

\noindent {\bf (c)}  
Since $k \in \Set$ then $\mu_k = r(e+1)(r+1)^{k-1} + \mu_{k-1}$ and, since $k+1
\in \Set$ then 
$\mu_{k+1} = r(e+1)(r+1)^{k} + \mu_{k} = r(e+1)(r+1)^{k-1}[(r+1)+1] +
\mu_{k-1}$. 
Recursively, for every $i=0, 1, \ldots, (k'-1)-k$ we have 
$$
\mu_{k+i} 
\, = \, 
r (e+1)(r+1)^{k-1} [(r+1)^i + (r+1)^{i-1} + \cdots + (r+1) + 1] + \mu_{k-1} 
\, = \, 
(e+1)(r+1)^{k-1} [(r+1)^{i+1} - 1] + \mu_{k-1} 
$$ 
and so 
$
|\mu|_k^{k'-1}  \, =  \, 
(e+1)(r+1)^{k-1} \sum_{i = 0}^{k'-1-k} [(r+1)^{j+1} - 1] +   \sum_{i = 0}^{k'-1-k} \mu_{k-1},
$ or 
$$
|\mu|_k^{k'-1} \, = \,  
 (e+1)(r+1)^{k - 1} [ \frac{(r+1)}{r} ( (r+1)^{k' - k} - 1 ) - (k' - k) ] 
 + (k' - k) \mu_{k-1} 
$$
and is easily rewritten in the desired form. 

\noindent {\bf (d)} follows in the same manner as {\bf (c)}, with the initial condition of $\mu_0 = e$. 
\end{proof}

If $(\alpha_i, n_i) := \{ \alpha_i, \alpha_i+1, \ldots, \alpha_i + n_i -1 \} \subseteq \Set$ then 
$(\alpha_i, n_i)$ is said to be a {\em chain in $\Set$}; it is said to be {\em maximal} if no other 
chain in $\Set$ properly contains $(\alpha_i, n_i)$. We can write a set $\Set$ in terms of the union of 
its maximal chains vis-a-vis $\Set = \{ (\alpha_1,n_1), (\alpha_2,n_2), \ldots, (\alpha_p,n_p)\}$.  
Since each of the chains are maximal then each $\alpha_i + n_i < \alpha_{i+1}$ and that 
$|\Set| = n_1+n_2+\cdots+n_p$. The following result provides us with the 1-norm of any vertex $\mu(\Set)$ 
in terms of the maximal chain description of $\Set$. 

\begin{lemma}\label{lem:gencase} 
For $s(\Set) = \{(\alpha_1,n_1),\ldots, (\alpha_p,n_p)\}$ with $\alpha_1>0$
\begin{equation} \label{eqn:non-zero-vertex-sum}
|\mu(\Set)| = \frac{(e+1)(r+1)^{n + 1 - |\Set|}}{r}
\Bigg[ 
(r+1)^{|\Set|}   
 - \prod_{i=1}^p \left(1+ \frac{r n_i}{r+1} \right)
\Bigg] . 
\end{equation}

For $s(\Set) = \{(0,n_1),\ldots, (\alpha_p,n_p)\}$   
\begin{equation}
|\mu(\Set)| = 
\frac{(e+1)(r+1)^{n + 1 - |\Set|}}{r}
\Bigg[
(r+1)^{|\Set|} 
 - \left(1 + \frac{n_1 r}{e+1}\right)
 \prod_{i=2}^p\left(1+ \frac{r n_i}{r+1}\right)
  \Bigg].
\end{equation}
\end{lemma}

\begin{proof} 
We prove both cases by induction on the number of maximal chains in $\Set$. 
For the base case we have $0 \notin \Set$ (i.e. $\alpha_1 > 0$) and $\Set = \{(\alpha_1,n_1)\}$. 
By setting $k = \alpha_1+n_1$ and $k'- 1 = n$ in Lemma~\ref{lem:runs}(a) the vertex $\mu$ for $\Set$ 
has norm given by 
$
|\mu| 
\, = \,  
(r+1)^{n+1-\alpha_1 - n_1}|\mu|^{\alpha_1+n_1-1}_{\alpha_1}.
$
The quantity $|\mu|^{\alpha_1+n_1-1}_{\alpha_1}$ can be computed from setting 
$k = \alpha_1$ and $k' = \alpha_1 + n_1$ in Lemma~\ref{lem:runs}(c) 
which verifies the base case when $\Set = \{ (\alpha_1, n_1)\}$ and $\alpha_1 > 0$: 
$$
|\mu| = (r+1)^{n+1-\alpha_1-n_1}|\mu|^{\alpha_1+n_1-1}_{\alpha_1} 
= \frac{(e+1)(r+1)^{(n+1- \alpha_1 - n_1) + \alpha_1}}{r} [(r+1)^{n_1} - 1 - (n_1) \frac{r}{r+1}] + (n_1)0 
$$
The base case when $0 \in \Set$ is similarly verified. 

Let us assume the induction hypothesis, that the claim is true for any set $\Set'$ with at most $p-1$ chains in it. 
Take a set $\Set =  \{(\alpha_1,n_1), \ldots, (\alpha_p, n_p)\}$ as a subset of $\{ 0,1,2,\ldots,n \}$.  
Since $k := \alpha_p + n_p, \ldots, k'-1 := n$ are not in $\Set$ then $1$-norm of the vertex for $\Set$ is given by 
$$
|\mu| 
= 
|\mu|^{\alpha_p + n_p -1} + |\mu|_{\alpha_p + n_p}^n 
=
(r+1)^{n+1-(\alpha_p + n_p)} |\mu|^{\alpha_p + n_p -1}
=
(r+1)^{n+1-(\alpha_p + n_p)} (|\mu|^{\alpha_p -1} + |\mu|_{\alpha_p}^{\alpha_p + n_p -1}) 
$$
Lemma~\ref{lem:runs}(c) with $k = \alpha_p$ and $k' = \alpha_p+n_p$ applies to $|\mu|_{\alpha_p}^{\alpha_p + n_p -1}$ 
to yield 
$$
|\mu|_{\alpha_p}^{\alpha_p + n_p -1} 
\, = \, 
\frac{(e+1)(r+1)^{\alpha_p + n_p}}{r}
\, - \, 
\frac{(e+1)(r+1)^{\alpha_p}}{r}
\, - \, 
n_p (e+1)(r+1)^{\alpha_p-1}
\, + \, 
n_p \frac{r}{r+1} |\mu|^{\alpha_p-1}
$$ 
so 
\begin{equation} \label{eqn:crank1}
|\mu|^{\alpha_p -1} + |\mu|_{\alpha_p}^{\alpha_p + n_p -1} 
\, = \, 
\frac{(e+1)(r+1)^{\alpha_p + n_p}}{r}
\, - \, 
\frac{(e+1)(r+1)^{\alpha_p}}{r}
\, - \, 
n_p (e+1)(r+1)^{\alpha_p-1}
\, + \, 
\Bigg( 1 + n_p \frac{r}{r+1} \Bigg) |\mu|^{\alpha_p-1}
\end{equation}
With $\Set' = \{(\alpha_1,n_1), \ldots, (\alpha_{p-1}, n_{p-1})\}$ viewed as a subset of $\{ 0,1,2,\ldots,\alpha_p - 1 \}$ 
the inductive hypothesis tells us that 
\begin{equation} \label{eqn:crank2}
|\mu|^{\alpha_p-1} = 
\frac{(e+1)(r+1)^{\alpha_p}}{r} - 
\frac{(e+1)(r+1)^{\alpha_p - |\Set'|}}{r}  
\prod_{i=1}^{p-1}\left(1+ \frac{r n_i}{r+1}\right)
\end{equation}
Substituting (\ref{eqn:crank2}) into (\ref{eqn:crank1}) 
\begin{align*}
|\mu|^{\alpha_p -1} + |\mu|_{\alpha_p}^{\alpha_p + n_p -1}
\, = \, 
\frac{(e+1)(r+1)^{\alpha_p + n_p}}{r}
\, - \, 
\frac{(e+1)(r+1)^{\alpha_p}}{r}
\, - \, 
n_p (e+1)(r+1)^{\alpha_p-1} & \\  
\, + \,
\left( 1 + n_p \frac{r}{r+1} \right) \frac{(e+1)(r+1)^{\alpha_p}}{r} 
- 
\left( 1 + n_p \frac{r}{r+1} \right) \frac{(e+1)(r+1)^{\alpha_p - |\Set'|}}{r} 
\prod_{i=1}^{p-1}\left(1+ \frac{r n_i}{r+1}\right)
 & \\ 
\end{align*}
with the second and third terms combined canceling with the fourth we have 
$$
|\mu|^{\alpha_p -1} + |\mu|_{\alpha_p}^{\alpha_p + n_p -1}
\, = \, 
\frac{(e+1)(r+1)^{\alpha_p + n_p}}{r} 
\, - \, 
\frac{(e+1)(r+1)^{\alpha_p - |\Set'|}}{r} 
\left( 1 + n_p \frac{r}{r+1} \right) 
\prod_{i=1}^{p-1}\left(1+ \frac{r n_i}{r+1}\right)
$$ 
and that last product is precisely $\prod_{i=1}^{p}\left(1+ \frac{r n_i}{r+1}\right)$. 
Finally, 
\begin{align*}
|\mu| &= 
(r+1)^{n+1-\alpha_p - n_p}(|\mu|^{\alpha_p -1} + |\mu|_{\alpha_p}^{\alpha_p + n_p -1})\\
	&=
\frac{(e+1)(r+1)^{n+1}}{r} -
\frac{(e+1)(r+1)^{n+1 - |\Set'| - n_p}}{r} \prod_{i=1}^{p}\left(1+ \frac{r n_i}{r+1}\right)
\end{align*}
as claimed. The case where $\alpha_1=0$ is proved in much the same way as the $\alpha_1 > 0$ case. 
\end{proof}

Given an arbitrary $\Set = \{(\alpha_1,n_1),\ldots, (\alpha_p,n_p)\}$ 
the previous lemma, with the exception of whether $\alpha_1$ equals $0$ or not, makes no reference 
to the values of the $\alpha_i$'s and we can thus compute the norms of all vertices of $P_n(e,r)$ with relative 
simplicity. Next are the only cases of $|\mu(\Set)|$ that we'll effectively need to compute our enumerating 
function for the lattice points of $P_n(e,r)$.  

\begin{corollary}\label{cor:mucases} For $\Set \subseteq \{0, \dots, n\}$: 
\textup{\bf (a)} If $\Set =\{j\}$ then 
$|\mu| = 
\left\lbrace 
\begin{array}{rcc}
(e)(r+1)^{n} & : &  j = 0 \\
r(e+1)(r+1)^{n-1} & : & j > 0 
\end{array} 
\right\rbrace .
$
\medskip

\noindent \textup{\bf (b)}  If $\Set = \{j_1,j_2\}$ then 
\medskip
$
|\mu|
= 
\left\lbrace 
\begin{array}{rcc}
 (e+1)(r+1)^{n} + \left[e-1\right](r+1)^{n-1} & : &  0=j_1=j_2 -1 \\
 (e+1)(r+1)^{n} + \left[r(e-1)-1)\right](r+1)^{n-2}& : & 0=j_1< j_2 -1 \\
(e+1)(r+1)^{n}+\left[r-1\right](e+1)(r+1)^{n-2} & : & 0 < j_1 = j_2 - 1 \\
(e+1)(r+1)^{n}+\left[r(r-1)-1\right](e+1)(r+1)^{n-3} & : & 0 < j_1 < j_2 - 1   
\end{array} 
\right\rbrace .
$

\medskip
\noindent \textup{\bf (c)} If $\Set = \{ j_1,j_2,j_3\}$ with $r=1$ then 
$
|\mu|
= 
\left\lbrace 
\begin{array}{rcc}
 (e+1)2^n + (7e -2) 2^{n-4} & : & 0 = j_1 < j_2-1 < j_3 -2  \\
 (e+1)2^n + (5e -1) 2^{n-3} & : & 0 = j_1 = j_2-1 < j_3 -2\\
 (e+1)2^{n}+ 7(e+1)2^{n-5} & : & 0 < j_1 < j_2-1 <  j_3 -2
 \end{array} 
\right\rbrace .
 $
\end{corollary}

\begin{proof} These are only special cases of Lemma~\ref{lem:gencase}. We prove 
part \textup{\bf (b)} when $0 < j_1 < j_2 - 1$ as an illustration. 
Lemma~\ref{lem:gencase} for $\Set = \{ (\alpha_1,n_1), (\alpha_2,n_2) \}$ with $\alpha_1>0$ 
and $n_1=n_2=1$ yields
\begin{align*}
|\mu| &= \frac{(e+1)(r+1)^{n+1-2}}{r}\left[(r+1)^2 - \left(1 + \frac{r}{(r+1)}\right)^2\right] 
	= \frac{(e+1)(r+1)^{n-1}}{r}\left[r^2 + 2r - \frac{2r}{(r+1)}-\frac{r^2}{(r+1)^2}\right] \\
	&= (e+1)(r+1)^{n-1}\left[r + 2 - \frac{2}{(r+1)}-\frac{r}{(r+1)^2}\right] 
	= (e+1)(r+1)^{n-1}\left[\frac{(r+2)(r+1)^2 - 2(r+1) - r}{(r+1)^2}  \right] \\
	&= (e+1)(r+1)^{n-3}\left[(r+1)^3 + (r+1)^2 - 2(r+1) - r\right]
	= (e+1)(r+1)^{n-3}\left[(r+1)^3 + r(r-1)-1\right]
\end{align*} 
which equals $(e+1)(r+1)^{n} + [r(r-1)-1](e+1)(r+1)^{n-3}$.
\end{proof}

Next we compute the primitive vectors associated with the vertices of $P_n(e,r)$ corresponding 
to $|\Set|\leq 2$. One straightforward consequence of Equation~(\ref{eqn:diff}) of Lemma~\ref{lem:unimodCones} 
(with the greatest common divisor term removed) is that we can describe primitive vectors recursively.  
Suppose $\TT = \Set \cup \{ j \}$. Then 
\begin{equation} \label{eqn:Prim-Vector}
\textup{p}_{\mu(\Set)}(\mu(\TT))_i = 
\left\lbrace 
\begin{array}{rcl}
0 & : &  i < j \\
1 & : &  i = j \\
\textup{p}_{\mu(\Set)}(\mu(\TT))_{i-1} & : &  i > j, \, i \in \TT \\
r \sum_{k=j}^{i-1} \textup{p}_{\mu(\Set)}(\mu(\TT))_{k} & : &  i > j, \, i \notin \TT 
\end{array} 
\right\rbrace
\end{equation}
This allows us to compute the $1$-norms of primitive vectors. In particular, when $\TT$ has one 
or two elements we get the following: 
\begin{corollary}\label{cor:primvecs} 
The primitive vectors associated with the vertex $\mu(\emptyset)$ are 
\begin{equation}
|\textup{p}_{\mu(\{ \emptyset \})}(\mu(\{ j \}))| = (r+1)^{n-j} \hspace{.1in} \textup{for each} \hspace{.1in} j=0,1,\ldots,n. 
\end{equation}
Those associated with the vertex $|\mu(\Set)|$ of $\Set=\{j_1,j_2\}$ where $j_1<j_2$ are 
\begin{equation} 
|\textup{p}_{\mu(\{k\})}(\mu(\Set))| = 
\left\lbrace 
\begin{array}{rcl}
(r+1)^{n-j_2} & : &   k =  j_1  \\ 
2(r+1)^{n-j_1-1} & : &  k = j_2 \, \textup{and } \, j_1 =  j_2-1 \\   
(2r+1)(r+1)^{n-j_1-2} & : & k = j_2 \, \textup{and } \, j_1 <  j_2-1    
\end{array} 
\right\rbrace
\end{equation}
If $\Set = \{j_1,j_2\}$ with $j_1<j_2-1$ then 
\begin{equation} 
|\textup{p}_{\mu(\Set)}(\mu(\Set \cup \{k\}))| = 
\left\lbrace 
\begin{array}{rcl}
(2r+1)^2(r+1)^{n-k-4} & : &  k < j_1-1 \\
2(2r+1)(r+1)^{n-k-3} & : &  k = j_1-1\\
(2r+1)(r+1)^{n-k-2} & : &  k < j_2 -1\\
2(r+1)^{n-k-1} & : & k = j_2 -1\\
(r+1)^{n-k} & : & k> j_2
\end{array} 
\right\rbrace
\end{equation}

\end{corollary}
\begin{proof}
In each case, the vectors $\mu(\TT)  -\mu(\Set)$ can be written from Equation~(\ref{eqn:Prim-Vector}) 
and the claims follow from the geometric series $1+r(1+(r+1)+\cdots+(r+1)^{k-1}) = (r+1)^k$.  
For example, 
$$
\textup{p}_{\mu(\emptyset)}(\mu(\{j\}))_i = 
\left\lbrace 
\begin{array}{rcl}
0 & : &  i < j \\
1 & : &  i = j \\
r(r+1)^{i-j-1} & : &  i > j 
\end{array} 
\right\rbrace
$$
and so $|\textup{p}_{\mu(\emptyset)}(\mu(\{j\}))| = 1 + r(1+ \ldots + (r+1)^{n-j-1}) = (r+1)^{n-j}$.
\end{proof}

We are now on the cusp of being able to prove our main results, the enumeration of the minimal 
relaxed complete partitions. One last remark and a preliminary lemma are needed. 

\begin{remark}\label{rem:prim}
As already noted in the proof of Lemma~\ref{lem:unimodCones}, $\Set \subseteq \TT \Rightarrow
|\mu(\Set)| \leq |\mu(\TT)|$ since each 
$\mu(\TT)_j - \mu(\Set)_j \geq 0$, being strictly positive if 
$j \in \TT \smallsetminus \Set $. 
\end{remark}

Recall from Section~\ref{sec:brion} that for a set $\TT \subseteq \{0,1,\ldots, n\}$ the \emph{order} 
of the vertex cone at $\mu(\TT)$ is 
$
O(\mu(\TT)) = |\mu(\TT)| + \sum_{\Set} |\textup{p}_{\mu(\Set)}(\mu(\TT))|
$ 
where the sum is taken over each $\Set = \TT \smallsetminus\{j\}$ for some $j \in \TT$. 
In the next lemma we will show that $O(\mu(\{0\})) \geq (e+1)(r+1)^n$. It is tempting to 
conclude, by virtue of set containment, that $O(\mu(\Set))$ also exceeds 
$(e+1)(r+1)^n$ whenever $0 \in \Set$ but it is not always true that $\Set \subseteq \TT$ 
implies $O(\mu(\Set)) \leq O(\mu(\TT))$ and hence the need for the different cases.

\begin{lemma} \label{lem:loose}
If $0 \in \Set$ and $|\Set| \leq 2$ then $O(\mu) \geq (e+1)(r+1)^n$. 
\end{lemma}
\begin{proof}
All claims regarding $|\mu|$ and the norms of primitive vectors can be read 
directly from Corollaries~\ref{cor:mucases} and \ref{cor:primvecs} respectively. 
If $\Set = \{0\}$ we have $|\mu| = e(r+1)^n$ and 
$|\textup{p}_{\mu( \emptyset )}(\mu(\{0\}))| = (r+1)^n$ 
and so 
$
O(\mu(\{0 \} ) ) \geq 
|\mu(\{ 0 \})| + |\textup{p}_{\mu( \emptyset )}(\mu(\{0\}))|  
= e(r+1)^{n} + (r+1)^{n}  = (e+1)(r+1)^{n}.
$ 

If $\Set = \{0,1\}$ then $|\mu| = (e+1)(r+1)^n + [e-1](r+1)^{n-1}$ 
and $|\textup{p}_{\mu( \{0\} )}(\mu(\{0,1\}))| = (r+1)^{n-1}$ and so 
$O(\mu) \geq   (e+1)(r+1)^n + [e-1](r+1)^{n-1} +  (r+1)^{n-1} = (e+1)(r+1)^n + e(r+1)^{n-1} 
\geq  (e+1)(r+1)^n$. 

If $\Set = \{0,j_2\}$ with $j_2>1$ then $|\mu| = (e+1)(r+1)^n - [r(e-1)-1](r+1)^{n-2}$ 
and $|\textup{p}_{\mu( \{j_2\} )}(\mu(\{0,j_2\}))| = (2r+1)(r+1)^{n-0-2}$ and so 
$O(\mu) \geq (e+1)(r+1)^n + [r(e-1)-1](r+1)^{n-2} + (2r+1)(r+1)^{n-2} 
= (e+1)(r+1)^n + r(e+1)(r+1)^{n-2}
\geq (e+1)(r+1)^n$. 
\end{proof}

\begin{theorem} \label{thm:r>1}
If $r \geq 2$ then 
  $\displaystyle \sigma_{P_n(e,r)}(x) 
   = 
   \sigma_{\KK_{\mu(\emptyset)}} + \sum_{j=1}^n \sigma_{\KK_{\mu(\{j\})}} + O(x^{(e+1)(r+1)^n})$
\end{theorem}

\begin{proof}
It will suffice to show that $O(\mu) \geq (e+1)(r+1)^n$ whenever 
{\bf(i)} $|\Set| \geq 2$ and $r \geq 2$, or 
{\bf (ii)} $0 \in \Set$. 
We can assume that the sets in (i) do not include $0$. 

\noindent {\bf (i):} 
If $\Set =\{j_1,j_2\}$ with $0< j_1 < j_2$ and $r \geq 2$ then, by 
Corollary~\ref{cor:mucases}(b), $|\mu(\Set)| \geq (e+1)(r+1)^n$. 
Given any other set $\TT$ of size $3$ or greater it must contain a set of the 
form $\Set =\{j_1,j_2\}$ and, by Remark~\ref{rem:prim}, 
$|\mu(\TT)| \geq |\mu(\Set)| \geq (e+1)(r+1)^n$. This settles (i).

\noindent {\bf (ii):} If $0 \in \Set$ with $|\Set| \geq 3$ then 
$\Set$ would contain a set of the form $\{ j_1,j_2\}$ with $j_1>0$ so, by 
case (i) and Remark~\ref{rem:prim}, these sets of size greater than $2$ that 
contain $0$ must also have $|\mu| \geq (e+1)(r+1)^n$. From Lemma~\ref{lem:loose}, 
we know that sets of size $1$ and $2$ that contain $0$ have the property that 
$O(\mu) \geq (e+1)(r+1)^n$ and so $O(\mu(\Set)) \geq (e+1)(r+1)^n$ for every set $\Set$ 
that contains $0$. 
\end{proof}

\begin{theorem}\label{thm:r=1}
If $r = 1$ then 
  $\displaystyle \sigma_{P_n(e,r)}(x) 
  = 
  \sigma_{\KK_{\mu(\emptyset)}} + \sum_{j=1}^n \sigma_{\KK_{\mu(\{j\})}} + 
  \sum_{j_1=1}^{n} \sum_{j_2 = j_1+2}^n \sigma_{\KK_{\mu(\{j_1, j_2\})}} + O(x^{(e+1)(r+1)^n})$
\end{theorem}

\begin{proof}
It will suffice to show that if $r = 1$ then $O(\mu) \geq (e+1)(r+1)^n$ 
whenever 
{\bf (i)} $\Set = \{j, j+1\}$,  
{\bf (ii)} $|\Set| \geq 3$, 
or 
{\bf (iii)} $0 \in \Set$. 
We can assume that the sets in (i) do not include $0$. 

\noindent {\bf (i):} 
Corollary~\ref{cor:mucases}(b) shows that if $\Set = \{j_1,j_2\}$ with $0< j_1=j_2-1$ then 
$|\mu| \geq (e+1)(r+1)^n$, for every $r \geq 1$.  
It follows from Remark~\ref{rem:prim} that any $\Set$ with two 
consecutive non-zero elements has a 1-norm $\geq (e+1)(r+1)^n$. 
In contrast to $r \geq 2$,  the case of $\Set = \{j_1,j_2\}$ with $0< j_1<j_2-1$ when $r=1$ has 
$
|\mu| 
= (e+1)2^{n}+\left[2(2-1)-1\right](e+1)2^{n-3} 
= (e+1)2^{n} - (e+1)2^{n-3} < (e+1)2^n
$ 
and, by consulting Corollary~\ref{cor:primvecs}, one can check that $O(\mu) < (e+1)2^n$. 

\noindent {\bf (ii):} 
If $|\Set| \geq 3$ then it may contain two consecutive elements in which case, 
by (i) and Remark~\ref{rem:prim}, $|\mu| \geq (e+1)(r+1)^n$. 
Alternatively, $\Set$ contains 
a set of the form $\{j_1,j_2,j_3 \}$ with $0<j_1<j_2-1<j_3-2$ and so 
$|\mu(\Set)| \geq |\mu(\{ j_1,j_2,j_3\})| \geq (e+1)2^n$ by Corollary~\ref{cor:mucases}(c). 

\noindent {\bf (iii):} Assume now that $0 \in \Set$. 
For $|\Set|\leq 2$ the result follows directly from Lemma~\ref{lem:loose}. 
For $|\Set|\geq 4$ then $\Set$ must contain a three element set that does not contain $0$ 
and so case (ii) combined with Remark~\ref{rem:prim} ensures that $|\mu(\Set)| \geq (e+1)2^n$. 
All that remains are the three element sets $\Set = \{0,j_2,j_3 \}$ with $j_2<j_3-1$. 

If $\Set = \{0,1,j_3\}$ then $|\mu| = (e+1)2^n + [5e-1]2^{n-3}$ 
and $|\textup{p}_{\mu( \{ 1,j_3 \} )}(\mu(\Set))| = 2(3) 2^{n-3}$ 
and so 
$O(\mu) 
\geq (e+1)2^n + [5e-1+6]2^{n-3} 
=  (e+1)2^n + 5(e+1)2^{n-3}
\geq (e+1)(r+1)^n$. 
Otherwise, $\Set = \{0,j_2,j_3\}$ with $j_2>1$ then 
$|\mu(\{0,j_1,j_2\})| = (e+1)2^n + (7e -2) 2^{n-4}$ 
and $|\textup{p}_{\mu( \{ j_2,j_3 \} )}(\mu(\Set))| = 3^2 2^{n-4}$
and so 
$O(\mu) 
\geq (e+1)2^n + 7(e+1)2^{n-4} 
\geq (e+1)(r+1)^n$. 
\end{proof}

Lastly, we present explicit descriptions of the above generating functions by computing 
$\sigma_{\KK_{\bf v}} $ for ${\bf v} =  \mu(\emptyset),  \mu(\{j_1\})$ and $ \mu(\{j_1,j_2\})$ 
with $1\leq j_1< j_2-1$. Reading directly fron Equation~(\ref{eqn:Flip}) 
and using all necessary 1-norms and primitive vectors from 
Corollaries~\ref{cor:mucases} and \ref{cor:primvecs} we have:

\noindent $\displaystyle 
\sigma_{\KK_{\mu(\emptyset)}} 
\, = \, 
x^0\prod_{j=0}^n \frac{1}{1-x^{|\textup{p}_{\mu(\emptyset)}(\mu(\{j\})|}} 
\, = \, \prod_{j=0}^n\frac{1}{1-x^{(r+1)^{j}}} 
$
and 
\begin{align*}
\sigma_{\KK_{\mu(\{j_1\})}} 
&=  
x^{r(e+1)(r+1)^{n-1}}
\frac{-x^{|\textup{p}_{\mu(\emptyset)}(\mu(\{j_1\})|} }
{1-x^{|\textup{p}_{\mu(\emptyset)}(\mu(\{j_1\})|} }
\prod_{k\not=j_1} \frac{1}{1-x^{|\textup{p}_{\mu(\{j_1\})}(\mu(\{j_1,k\})|}} \\ 
& = 
x^{r(e+1)(r+1)^{n-1}}
\frac{-x^{(r+1)^{n-j_1}}}{1-x^{(r+1)^{n-j_1}}}
\frac{1}{1-x^{2(r+1)^{n-j_1}}} 
\prod_{k=j_1+1}^n \frac{1}{1-x^{(r+1)^{n-k}}}
\prod_{k=0}^{j_1-2} \frac{1}{1-x^{(2r+1)(r+1)^{n-k-2}}}\\
& = 
-\frac{x^{r(e+1)(r+1)^{n-1} + (r+1)^{n-j_1}}}
{1-x^{(r+1)^{n-j_1}} 1-x^{2(r+1)^{n-j_1}}}
\prod_{t=0}^{n-j_1-1} \frac{1}{1-x^{(r+1)^{t}}} 
\prod_{t=n-j_1}^{n-2} \frac{1}{1-x^{(2r+1)(r+1)^{t}}}, 
\end{align*}

For $\Set =  \{j_1,j_2\}$ with $0< j_1 < j_2-1$ and $r=1$
\begin{align*}
\sigma_{\KK_{\mu(\{j_1,j_2\})}} 
& =  
x^{(e+1)7 \cdot 2^{n-3}}
\prod_{k\in\{j_1,j_2\}}
\frac{-x^{|\textup{p}_{\mu(\{k\})}(\mu(\{j_1,j_2\})|}}
{1-x^{|\textup{p}_{\mu(\{k\})}(\mu(\{j_1,j_2\})|}}
\prod_{k\not\in\{j_1,j_2\}}
\frac{1}{1-x^{|\textup{p}_{\mu(\{j_1,j_2\})}(\mu(\{j_1,j_2,k\})|}} \\
& =
x^{(e+1)7 \cdot 2^{n-3}} x^{2^{n-j_2}}x^{3 \cdot 2^{n-j_1 -2}}
\prod_{t=0}^{n-j_2+1}
\frac{1}{1-x^{2^t}} \prod_{t = n-j_2}^{n-j_1 -1}
\frac{1}{1-x^{3 \cdot 2^t}}
\prod_{t = n - j_1 -2}^{n - 4} \frac{1}{1-x^{9 \cdot 2^t}}.
\end{align*}

Note that when $(e,r)=(1,2)$ and $n=2$ we recover the enumerating function $\sigma_{P_2(1,2)}(x)$ 
as described in the introduction. Setting $e=0$ we recover R{\o}dseth's \cite{RodMpart, Rod} 
enumerating functions for the minimal complete partitions.

\end{document}